\documentclass[11pt]{amsart}

\title[Haken spheres for genus two Heegaard splittings]
{Haken spheres for genus two Heegaard splittings}

\author{Sangbum Cho}
\thanks{The first-named author is supported by Basic Science Research Program through the National Research Foundation of Korea (NRF) funded by the Ministry of Science, ICT and Future Planning (NRF-2015R1A1A1A05001071).}
\address{Department of Mathematics Education  \newline
\indent Hanyang University, Seoul 133-791, Korea}
\email{scho@hanyang.ac.kr}

\author{Yuya Koda}
\thanks{The second-named author is supported in part
by the Grant-in-Aid for Young Scientists (B), JSPS KAKENHI Grant Number 26800028.}

\address{
Department of Mathematics \newline
\indent Hiroshima University, 1-3-1 Kagamiyama, Higashi-Hiroshima, 739-8526, Japan}
\email{ykoda@hiroshima-u.ac.jp}

\subjclass[2000]{Primary 57N10; 57M60.}

\date{\today}

\usepackage{amsfonts,amsmath,amssymb,amscd}

\usepackage{amsthm}
\usepackage{latexsym}
\usepackage[dvips]{graphicx}
\usepackage[dvips]{psfrag}
\usepackage[dvips]{color}
\usepackage{xypic}
\usepackage[abs]{overpic}
\usepackage{caption}
\usepackage[all]{xy}

\theoremstyle{plain}
\newtheorem*{theorem*}{Theorem}
\newtheorem*{lemma*} {Lemma}
\newtheorem*{corollary*} {Corollary}
\newtheorem*{proposition*}{Proposition}
\newtheorem*{conjecture*}{Conjecture}
\newtheorem{theorem}{Theorem}[section]
\newtheorem{lemma}[theorem]{Lemma}

\theoremstyle{remark}

\newtheorem*{remark}{Remark}

\theoremstyle{definition}

\newcommand{\Nbd}{\operatorname{Nbd}}



\begin{document}
\maketitle

\begin{abstract}
A manifold which admits a reducible genus-$2$ Heegaard splitting is one of the $3$-sphere, $S^2 \times S^1$, lens spaces or their connected sums.
For each of those splittings, the complex of Haken spheres is defined.
When the manifold is the $3$-sphere, $S^2 \times S^1$ or the connected sum whose summands are lens spaces or $S^2 \times S^1$, the combinatorial structure of the complex has been studied by several authors. In particular, it was shown that those complexes are all contractible.
In this work, we study the remaining cases, that is, when the manifolds are lens spaces.
We give a precise description of each of the complexes for the genus-$2$ Heegaard splittings of lens spaces.
A remarkable fact is that the complexes for most lens spaces are not contractible and even not connected.
\end{abstract}

%\vspace{1em}

%\begin{small}
%\hspace{2em}  \textbf{2010 Mathematics Subject Classification}:
%57M15, 05C10; 57M25, 05C25, 57S05

%57M15 Relations with graph theory
%05C10  Planar graphs; geometric and topological aspects of graph theory [See also
%57M25 Knots and links in $S^3$
%05C25 Graphs and abstract algebra (groups, rings, fields, etc.)
%57S05 Topological properties of groups of homeomorphisms or diffeomorphisms

%\hspace{2em}
%\textbf{Keywords}:
%topological symmetry group; mapping class group; spatial graph; group of
%
%\hspace{2em} homeomorphisms; tunnel number one knot; tunnel.
%\end{small}

\section*{Introduction}

Let $(V, W; \Sigma)$ be a genus-$2$ Heegaard splitting of a closed orientable 3-manifold $M$.
That is, $M = V \cup W$ and $ V \cap W = \partial V =\partial W =\Sigma$, where $V$ and $W$ are genus-$2$ handlebodies.
A {\it Haken sphere} $S$ for the splitting $(V, W; \Sigma)$ is a separating sphere in $M$ that intersects $\Sigma$ transversely in a single essential circle.
The circle $S \cap \Sigma$ is necessarily separating in $\Sigma$.
Two Haken spheres $S$ and $T$ are said to be {\it equivalent} if $S \cap \Sigma$ is isotopic to $T \cap \Sigma$ in $\Sigma$. 
We note that the study of Haken spheres of $(V, W; \Sigma)$ corresponds to 
the study of tunnel number-1 links in $M$. 

Given two Haken spheres $S$ and $T$ for the splitting $(V, W; \Sigma)$, the {\it intersection number} $S \cdot T$ is defined to be the minimal number of points of $S \cap T \cap \Sigma$ up to isotopy in $\Sigma$.
Then $S$ and $T$ are equivalent if and only if $S \cdot T =0$, and otherwise we have $S \cdot T \geq 4$ by Scharlemann-Thompson \cite{ST03}.
When $S \cdot T = 4$, we say that $S$ and $T$ are joined by a {\it $4$-gon replacement} or by an {\it elementary move}.
The notion of $4$-gon replacement proposes a natural simplicial complex called the {\it complex of Haken spheres}, or simply the {\it sphere complex}, for the splitting $(V, W; \Sigma)$.
The sphere complex is defined as follows.
The vertices of the sphere complex are the equivalence classes of Haken spheres, and a collection of distinct $k + 1$ vertices $S_0, S_1, \ldots , S_k$ spans a $k$-simplex if and only if $S_i \cdot S_j = 4$ for all $0 \leqslant i < j \leqslant k$.

If a genus-$2$ Heegaard splitting of a manifold admits a Haken sphere, then the manifold is one of the $3$-sphere, $S^2 \times S^1$, lens spaces and their connected sums.
We note that the genus-$2$ Heegaard splittings of these manifolds are
completely classified
up to isotopy by Waldhausen \cite{Wal68}, Bonahon-Otal \cite{BO83} and \cite{CK13}.
In fact, among them, each prime manifold admits a unique genus-$2$ Heegaard splitting,
while each non-prime one admits at most two genus-$2$ Heegaard splittings.
It is natural to ask whether the structures, such as dimension, connectivity and contractibility,
of the sphere complex of each genus-$2$ Heegaard splitting
reflect the topological structure of the corresponding $3$-manifold. 
In \cite{Sch04}, Scharlemann showed that the sphere complex for the genus-$2$ Heegaard splitting of the $3$-sphere is connected, and subsequently a complete description of its shape was given in Akbas \cite{Akb08} and Cho \cite{Cho08}. 
In fact, this sphere complex is $2$-dimensional contractible complex and each edge is contained in a single $2$-simplex.
In Lei \cite{Lei05} and Lei-Zhang \cite{LZ04}, it was shown that the sphere complex is connected for a genus-$2$ Heegaard splitting of the connected sum whose summands are lens spaces or $S^2 \times S^1$.
In \cite{CK13} later, this result is refined that they are all contractible. In fact, the sphere complex is a tree if the summands are both lens spaces or both $S^2 \times S^1$, and is a contractible $3$-dimensional complex if one summand is $S^2 \times S^1$ and
the other one a lens space.
In \cite{CK14}, it is shown that the sphere complex is a $3$-dimensional contractible complex when the maninfold is $S^2 \times S^1$.

In this paper, we investigate the sphere complexes for the remaining cases, that is, for the genus-$2$ Heegaard splittings of lens spaces, and give a description of each of those complexes.
We state the main result as follows.

\begin{theorem}
\label{thm:main theorem}
Let $(V, W; \Sigma)$ be the genus-$2$ Heegaard splitting of a lens space $L = L(p, q)$ with $1 \leq q \leq p/2$, and let $\mathcal S(V, W; \Sigma)$ be the sphere complex for $(V, W; \Sigma)$.
\begin{enumerate}
  \item $\mathcal S(V, W; \Sigma)$ contains no $3$-cycle, that is, it is a $1$-dimensional complex. Each vertex of $\mathcal S(V, W; \Sigma)$ has infinite valency.
  \item $\mathcal S(V, W; \Sigma)$ is connected if and only if $L$ is $L(p, 1)$.
  \item If $L$ is $L(2, 1)$, then every edge in $\mathcal S(V, W; \Sigma)$ is contained in a unique cycle, which is a $4$-cycle.
  \item If $L$ is $L(3, 1)$, then every edge in $\mathcal S(V, W; \Sigma)$ is contained in a unique cycle, which is a $6$-cycle.
  \item If $L$ is $L(p, 1)$ with $p \geq 4$, then $\mathcal S(V, W; \Sigma)$ is a tree.
  \item If $L$ is not $L(p, 1)$, then $\mathcal S(V, W; \Sigma)$ is not connected, and it consists of infinitely many tree components.
\end{enumerate}
\end{theorem}

\bigskip

\begin{center}
\begin{overpic}[width=15cm, clip]{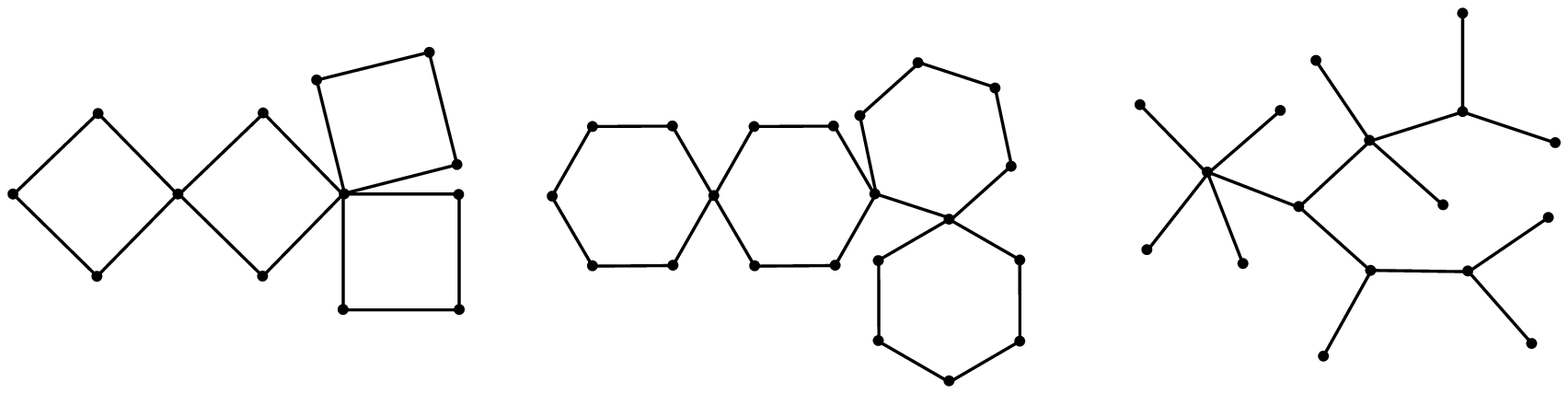}
  \linethickness{3pt}
    \put(50,10){$L(2, 1)$}
    \put(180,10){$L(3, 1)$}
    \put(300,10){$L(p, 1)$ with $p \geq 4$}
\end{overpic}
\captionof{figure}{}
\label{fig:sphere_complexes}
\end{center}

\medskip

Figure \ref{fig:sphere_complexes} illustrates a small portion of the complex $\mathcal S(V, W; \Sigma)$ for each case.
In the theorem, an $n$-cycle, $n \geq 3$, means a $1$-dimensional subcomplex of $\mathcal S(V, W; \Sigma)$ consisting of exactly $n$ distinct edges and $n$ distinct vertices such that each vertex has valency two.
It is remarkable that the sphere complex is not connected for every lens space other than $L(p, 1)$, and the complexes for $L(2, 1)$ and $L(3, 1)$ are connected but not contractible.
The idea of the proof is based on the following fact.
Given the genus-$2$ Heegaard splitting of each of lens spaces, there is a one-to-one correspondence between the collection of Haken spheres and the collection of pairs of a primitive disk and one of its dual disks up to isotopy.
Then the $4$-gon replacement of a Haken sphere to another one can be translated
in terms of such pairs in a simple way, and then we can use the known results on the primitive disks developed in \cite{Cho12} and \cite{CK15}.

We use the standard notation $L = L(p, q)$ for a lens space in standard textbooks.
For example, we refer \cite{Ro} to the reader.
Throughout the paper, $(V, W; \Sigma)$ will denote the genus-$2$ Heegaard splitting of a lens space $L(p, q)$, and we will always assume $1 \leq q \leq p/2$.
We will denote by $\mathcal S(V, W; \Sigma)$ the sphere complex for the splitting $(V, W; \Sigma)$.
For convenience, we will not distinguish (sub)spaces and homeomorphisms from their isotopy classes in their notation.

\section{The primitive disks and their dual disks}
\label{sec:primitive_disks}

An essential disk $D$ in $V$ is said to be {\it primitive}
if there exists an essential disk $D^\prime$ in $W$
such that $\partial D$ intersects $\partial D^\prime $ transversely in a single point.
Such a disk $D^\prime$ is called a {\it dual disk} of $D$, and we call the ordered pair $(D, D^\prime)$ simply a {\it dual pair}.
We note that the disk $D^\prime$ is also primitive in $W$ with a dual disk $D$, and that $W \cup \Nbd(D)$ and $V \cup \Nbd(D^\prime)$ are solid tori.
Primitive disks are necessarily non-separating.
We call a pair $\{D, E\}$ of disjoint primitive disks in a handlebody simply a {\it primitive pair}, and similarly a triple of pairwise disjoint primitive disks a {\it primitive triple}.
When the two disks $D$ and $E$ in a primitive pair $\{D, E\}$ have a common dual disk, we simply say that the pair $\{D, E\}$ admits a common dual disk.
Of course, there exist infinitely many primitive pairs of the handlebodies $V$ and $W$ in the genus-$2$ Heegaard splitting of a lens space.
In fact, any primitive disk is contained in infinitely many primitive pairs.
However not every genus-$2$ Heegaard splitting of a lens space
admits a primitive triple.
See Lemma \ref{lem:primitive_triples} in the following.

Given a dual pair $(D, D^\prime)$, the boundary of a regular neighborhood of $D \cup D^\prime$ is a Haken sphere for the splitting $(V, W; \Sigma)$.
Conversely, if $S$ is a Haken sphere for the splitting $(V, W; \Sigma)$ of a lens space, then $S$ cuts off a 3-ball $B$ from the lens space.
Set $V_1 = V \cap B$ and $W_1 = W \cap B$, then both $V_1$ and $W_1$ are solid tori, and there are meridian disks $D$ and $D^\prime$ of $V_1$ and $W_1$ respectively, such that $D$ and $D^\prime$ are disjoint from $\partial B$, and $\partial D$ intersects $\partial D^\prime$ transversely in a single point.
The disks $D$ and $D^\prime$ are unique up to isotopy, and they form a dual pair $(D, D^\prime)$ of the splitting $(V, W; \Sigma)$.
Thus there is a one-to-one correspondence between the equivalence classes of Haken spheres and the isotopy classes of dual pairs.
Furthermore, we have the following lemma immediately.

\begin{lemma}
\label{lem:reducing spheres and dual pairs}
Let $S$ and $T$ be the Haken spheres corresponding to
the dual pairs $(D, D')$ and $(E, E')$, respectively, of the splitting $(V, W; \Sigma)$,
where $D, E \subset V$ and $D', E' \subset W$.
Then $S \cdot T = 4$ if and only if, up to isotopy, either
\begin{itemize}
  \item
  $\{D, E\}$ is a primitive pair of $V$, and $D'$ is equal to $E'$ and is a common dual disk of $\{D, E\}$, or
  \item
  $\{D', E'\}$ is a primitive pair of $W$, and $D$ is equal to $E$ and is a common dual disk of $\{D', E'\}$.
\end{itemize}
\end{lemma}

\begin{proof}
The ``if'' part is clear. For the ``only if'' part, we assume $S \cdot T = 4$.
The sphere $S$ cuts off a $3$-ball $B$ from the lens space, and we have four solid tori $V_1 = V \cap B$, $W_1 = W \cap B$, $V_2 = \overline{V - V_1}$ and $W_2 = \overline{W - W_1}$.
For the dual pair $(D, D')$ corresponding to $S$, $D$ and $D'$ are meridian disks of $V_1$ and $W_1$ respectively.
Since $S \cdot T = 4$, the disk $S \cap V$ divides the disk $T \cap V$ into two bigons $\Delta_1$ and $\Delta_2$ with a rectangle $R$.
Similarly, $S \cap W$ divides the disk $T \cap W$ into two bigons $\Delta'_1$ and $\Delta'_2$ with a rectangle $R'$.
Both of the bigons $\Delta_1$ and $\Delta_2$ are contained in either $V_1$ or $V_2$.

\begin{center}
\begin{overpic}[width=14cm, clip]{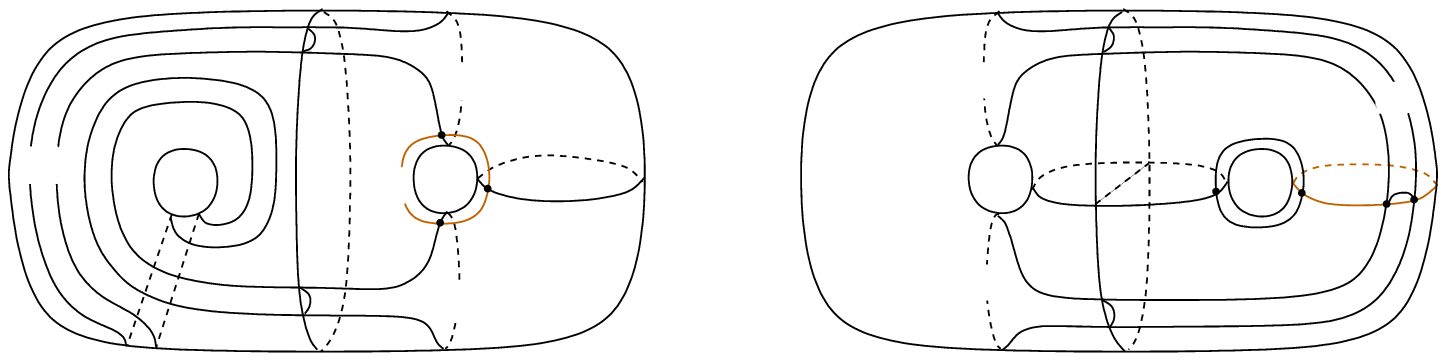}
  \linethickness{3pt}
    \put(20,10){$V_2$}
    \put(170,10){$V_1$}
    \put(220,10){$W_2$}
    \put(380,10){$W_1$}
    \put(80,0){$S \cap V$}
    \put(290,0){$S \cap W$}
    \put(18,62){$R$}
    \put(130,85){$\Delta_1$}
    \put(130,30){$\Delta_2$}
    \put(150,47){$D$}
    \put(101,54){$\partial D'$}
    \put(370,78){$R'$}
    \put(260,85){$\Delta'_1$}
    \put(260,30){$\Delta'_2$}
    \put(360,45){$D'$}
    \put(315,45){$E'$}
    \put(335,75){$\partial D$}
\end{overpic}
\captionof{figure}{}
\label{fig:sphere_and_dual_pair}
\end{center}

First, consider the case that $\Delta_1$ and $\Delta_2$ are contained in $V_1$
(see Figure \ref{fig:sphere_and_dual_pair}).
Then they are meridian disks of $V_1$, which are disjoint from the meridian disk $D$ of $V_1$
up to isotopy.
Since the circle $\partial D$ is a longitude of the solid torus $W_1$, the rectangle $R'$ of the disk $T \cap W$ is contained in the solid torus $W_1$, and is disjoint from $\partial D$, and intersects $D'$ in a single arc.
Consequently, the two bigons $\Delta'_1$ and $\Delta'_2$ are contained in $W_2$, and hence the disk $T \cap W$ is disjoint from $\partial D$ and intersects $D'$ in a single arc.
The disk $T \cap W$ divides $W$ into two solid tori.
Let $W'_1$ be one of them containing $\partial D$ as a longitude.
Then the disk $E'$ of the dual pair $(E, E')$ is a meridian disk of $W'_1$, and, up to isotopy, $E'$ is disjoint from $D'$ and intersects the longitude $\partial D$ in a single point.
Thus we have that  $\{D', E'\}$ is a primitive pair of $W$, and $D$ is equal to $E$ and is a common dual disk of $\{D', E'\}$.
If $\Delta_1$ and $\Delta_2$ are contained in $V_2$, we conclude that $\{D, E\}$ is a primitive pair of $V$, and $D'$ is equal to $E'$ and is a common dual disk of $\{D, E\}$ by a similar argument.
\end{proof}

By the lemma, we can translate the properties of Haken spheres into those of dual pairs.
The properties of primitive disks and their dual disks that we need were already developed in
\cite{CK15}, which are summarized in the following two lemmas.

\begin{lemma}[Theorem 4.3 \cite{CK15}]
\label{lem:common_duals}
Given a lens space $L = L(p, q)$, $1 \leq q \leq p/2$, with a genus-$2$ Heegaard splitting $(V, W; \Sigma)$, each primitive pair in $V$ has a common dual disk if and only if $q = 1$.
In this case, if $p \geq 3$, the pair has a unique common dual disk, and if $p = 2$, the pair has  exactly two disjoint common dual disks, which form a primitive pair in $W$.
\end{lemma}

If $L$ is not $L(p, 1)$, then any primitive pair in $V$ admits either a unique common dual disk or no common dual disk up to isotopy. Further, any primitive disk is contained in infinitely many primitive pairs having a common dual disk and simultaneously in infinitely many primitive pairs having no common dual disk up to isotopy.

\begin{lemma}[Theorem 4.4 \cite{CK15}]
\label{lem:primitive_triples}
Given a lens space $L(p, q)$, for $1 \leq q \leq p/2$, with a genus-$2$ Heegaard splitting $(V, W; \Sigma)$ of $L(p, q)$, there is a primitive triple in $V$ if and only if $q = 2$ or $p= 2q +1$.
In this case, we have the following refinements.
\begin{enumerate}
\item If $p = 3$, then each primitive pair is contained in a unique primitive triple.
\item If $p = 5$, then each primitive pair having a common dual disk is contained in a unique primitive triple, and each having no common dual disk is contained in exactly two primitive triples.
\item If $p \geq 7$, then each primitive pair having a common dual disk is contained either in a unique or in no primitive triple, and each having no common dual disk is contained in a unique primitive triple.
\item Further, if $p = 3$, then each of the three primitive pairs in any primitive triple in $V$ has a unique common dual disk, which form a primitive triple in $W$.
    If $p \geq 5$, then exactly one of the three primitive pairs in any primitive triple has a common dual disk, which is unique.
\end{enumerate}
\end{lemma}

Recalling the fact that $V$ and $W$ are isotopic in $L(p,q)$ shown by
Bonahon-Otal \cite{BO83}, it is clear that
the above two lemmas still hold when we exchange $V$ and $W$ in the statements.

\section{The complex of primitive disks}
\label{sec:primitive_disk_complex}

Given an irreducible $3$-manifold with compressible boundary, the {\it disk complex} of the manifold is a simplicial complex defined as follows.
The vertices are the isotopy classes of essential disks in the manifold, and a collection of $k+1$ vertices spans a $k$-simplex if and only if it admits a collection of representative disks which are pairwise disjoint.
It is well known from McCullough \cite{McC91}, the disk complex for any irreducible $3$-manifold with compressible boundary is contractible, and further, in
Cho \cite{Cho08}, a useful criterion was developed to determine whether a given subcomplex of a disk complex is contractible or not.
We will not introduce the details here, but just summarize the results we need.

Consider the case that the manifold is a genus-$2$ handlebody $V$.
Then the disk complex of $V$ is a $2$-dimensional contractible complex.
We denote by $\mathcal D(V)$ the {\it non-separating disk complex} of $V$, which is the full subcomplex of the disk complex spanned by the vertices of non-separating disks.
It is easy to see that $\mathcal D(V)$ is also $2$-dimensional and every edge of $\mathcal D(V)$ is contained in infinitely but countably many $2$-simplices.
Further, the complex $\mathcal D(V)$ is contractible, and the link of any vertex of $\mathcal D(V)$ is also contractible, i.e. the link is a tree, see \cite{McC91}, \cite{Cho08}.
Thus, we can describe the structure of the non-separating disk complex $\mathcal D(V)$;
a portion of $\mathcal D(V)$ is described  in Figure \ref{disk_complex}.
We observe that $\mathcal D(V)$ deformation retracts to a tree in the barycentric subdivision of it.
Actually, this tree is a dual complex of $\mathcal D(V)$.
From the structure of $\mathcal D(V)$, every connected component of any full subcomplex of $\mathcal D(V)$ is contractible.

\begin{remark}
Here is a simple but important observation on the structure of $\mathcal D(V)$.
If $\mathcal C$ is any $n$-cycle in $\mathcal D(V)$ with $n \geq 4$, then there is no $2$-simplex of which the three edges are all contained in $\mathcal C$.
But there exist at least two $2$-simplices in $\mathcal D(V)$ such that exactly two edges of each of the two $2$-simplices are contained in $\mathcal C$ and further the two $2$-simplices have no common edges contained in $\mathcal C$.
\end{remark}

\begin{center}
\begin{overpic}[width=7cm, clip]{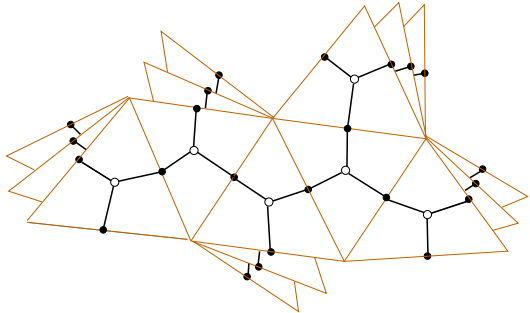}
  \linethickness{3pt}
\end{overpic}
\captionof{figure}{A portion of the non-separating disk complex $\mathcal D(V)$ of a genus-$2$ handlebody $V$ with its dual complex.}
\label{disk_complex}
\end{center}

\bigskip

Now we return to the genus-$2$ Heegaard splitting $(V, W; \Sigma)$ of a lens space $L = L(p, q)$ with $1 \leq q \leq p/2$.
The {\it primitive disk complex} $\mathcal P(V)$ for the splitting $(V, W; \Sigma)$ is defined to be the full subcomplex of $\mathcal D(V)$ spanned by the vertices of primitive disks in $V$.
The primitive disk complexes and their variations are studied
in several settings in \cite{CK}, \cite{CK13}, \cite{CK15} and \cite{Kod}.
The complete combinatorial structure of the primitive disk complex $\mathcal P(V)$ for each lens space has been well understood (see \cite[Theorem 4.5 and Figure 11]{CK15}).
In particular, $\mathcal P(V)$ is contractible if $p \equiv \pm 1 \pmod q $, and otherwise it consists of infinitely many tree components.
By Lemma \ref{lem:primitive_triples}, the complex $\mathcal P(V)$ has a $2$-simplex if and only if $q = 2$ or $p = 2q + 1$.

We define two special subcomplexes,
one is of $\mathcal P(V)$ and the other is of $\mathcal{P}(W)$.
First, $\mathcal P'(V)$ is the subcomplex of $\mathcal P(V)$ defined as follows.
The vertices of $\mathcal P'(V)$ are the vertices of $\mathcal P(V)$, and distinct $k+1$ vertices $D_0, D_1, \cdots, D_k$ of $\mathcal P'(V)$ span a $k$-simplex if and only if the primitive pair $\{D_i, D_j\}$ admits a common dual disk for each $0 \leq i < j \leq k$.
Next, given a primitive disk $D$ in $V$, we denote by $\mathcal P_D(W)$ the full subcomplex of the primitive disk complex $\mathcal P(W)$ spanned by the vertices of dual disks of $D$.

\begin{lemma}
\label{lem:primitive_disk_complex}
Let $(V, W; \Sigma)$ be a genus-$2$ Heegaard splitting of a lens space $L = L(p, q)$ with $1 \leq q \leq p/2$.
\begin{enumerate}
  \item If $L$ is $L(p, 1)$ with $p \neq 3$, then $\mathcal P'(V)$ equals $\mathcal P(V)$, which is a tree.
  \item If $L$ is $L(3, 1)$, then $\mathcal P'(V)$ equals $\mathcal P(V)$, a contractible $2$-dimensional complex, each of whose edge is contained in a unique $2$-simplex.
  \item If $L$ is not $L(p, 1)$, then $\mathcal P'(V)$ is not connected, and it consists of infinitely many tree components.
\end{enumerate}
\end{lemma}

Lemma \ref{lem:primitive_disk_complex} is a direct consequence of
Lemma \ref{lem:primitive_triples} (\cite[Theorem 4.4]{CK15}) and \cite[Theorem 5.5]{CK15}.

\begin{lemma}
\label{lem:complex_of_dual_disks}
Let $(V, W; \Sigma)$ be the genus-$2$ Heegaard splitting of a lens space $L = L(p, q)$ with $1 \leq q \leq p/2$. Then given any primitive disk $D$ in $V$, the complex $\mathcal P_D(W)$ is a tree whose vertices have infinite valency.
\end{lemma}

It is easy to see that $\mathcal P_D(W)$ is $1$-dimensional since there is no primitive triple in $W$ whose vertices are represented by dual disks of $D$.
The contractibility of $\mathcal P_D(W)$ is proved using Theorem 2.1 in \cite{CK15} by a similar but simpler argument to the contractibility of $\mathcal P(V)$ for $L(p, 1)$. That is, when two dual disks $D'$ and $E'$ of $D$ intersect each other, one of the two disks from surgery on
$D'$ along an outermost subdisk of $D''$ cut off by $D' \cap D''$ is again a dual disk of $D$.

\section{Proof of the main theorem}
\label{sec:Proof of the main theorem}

In this section, we prove Theorem \ref{thm:main theorem}.
Let $(V, W; \Sigma)$ be the genus-$2$ Heegaard splitting of a lens space $L(p, q)$ with $1 \leq q \leq p/2$, and let $\mathcal S(V, W; \Sigma)$ be the sphere complex for $(V, W; \Sigma)$.

\bigskip

\noindent (1) $\mathcal S(V, W; \Sigma)$ contains no $3$-cycle, that is, it is a $1$-dimensional complex. Each vertex of $\mathcal S(V, W; \Sigma)$ has infinite valency.

\medskip

It is clear
from Lemmas \ref{lem:reducing spheres and dual pairs} and
\ref{lem:complex_of_dual_disks} that each vertex of $\mathcal S(V, W; \Sigma)$ has infinite valency.
Suppose that there exists a $2$-simplex in $\mathcal S(V, W; \Sigma)$.
Let $S$, $T$ and $R$ be the three vertices of the $2$-simplex.
We may assume that the dual pairs corresponding to $S$ and $T$ are $(D, D')$ and $(E, D')$, respectively.
Then the only possible dual pair corresponding to $R$ is $(F, D')$ for some dual disk $F$ of $D'$.
Then we have the primitive triple $\{D, E, F\}$ of dual disks of $D'$, which contradicts the fact that $\mathcal P_{D'}(V)$ is a tree in Lemma \ref{lem:complex_of_dual_disks}.

\bigskip

\noindent (2) $\mathcal S(V, W; \Sigma)$ is connected if and only if $L$ is $L(p, 1)$.

\medskip

Assume that $L$ is $L(p, 1)$, and let $S$ and $T$ be any two distinct vertices of $\mathcal S(V, W; \Sigma)$.
We may assume that the Haken spheres $S$ and $T$ correspond to the dual pairs $(D, D')$ and $(E, E')$, where $D$ and $E$ are primitive disks in $V$, and $D'$ and $E'$ are their dual disks respectively.
By Lemma \ref{lem:primitive_disk_complex} (1) and (2), the complex $\mathcal P'(V)$ is connected, and hence there exists a sequence $D = D_0, D_1, D_2, \cdots, D_n = E$ of primitive disks in $V$ such that each of $\{D_{i-1}, D_i\}$ is a primitive pair of $V$ having a common dual disk $D'_i$ for each $i \in \{1, 2, \cdots, n\}$.
We set $D_0' = D'$ and $D_{n+1}' = E'$.

For each $i \in \{ 0, 1, \ldots, n \}$, two dual disks $D'_i$ and $D'_{i+1}$ of $D_i$ might intersect each other, and so $\{D'_i, D'_{i+1}\}$ may not be a primitive pair.
But, since $\mathcal P_{D_i}(W)$ is a tree by Lemma \ref{lem:complex_of_dual_disks},
there exists a sequence of dual disks
$D_i' = D'_{i,0}, D'_{i,1}, \cdots, D'_{i, n_i}, D'_{i, n_i+1} = D_{i+1}'$ of $D_i$
such that each of $\{D'_{i,j}, D'_{i,j+1}\}$ is a primitive pair of $W$ for each $j \in \{0, 1, \cdots, n_i\}$.
Thus, by Lemma \ref{lem:reducing spheres and dual pairs}, we obtain a sequence of dual pairs
$(D_i, D_{i}') = (D_i, D_{i,0}'), (D_i, D_{i,1}'), \ldots,
(D_i, D_{i,n_i+1}') = (D_i, D_{i+1}')$,
which realizes a path in $\mathcal S(V, W; \Sigma)$ from the vertex of $(D_i, D_{i}')$ to of $(D_i, D_{i+1}')$ for each $i \in \{ 0, 1, \ldots, n \}$.
Moreover, the two vertices corresponding to the dual pairs $(D_i, D'_{i+1})$ and $(D_{i+1}, D'_{i+1})$ are joined by an edge in $\mathcal S(V, W; \Sigma)$ for each  $i \in \{ 0, 1, \ldots, n \}$.
Consequently, we obtain a sequence of dual pairs from $(D, D')$ to $(E, E')$ which realizes a path from $S$ to $T$ in $\mathcal S(V, W; \Sigma)$.

Conversely, assume that $L$ is not $L(p, 1)$.
Then, by Lemma \ref{lem:common_duals}, there exists a primitive pair $\{D, E\}$ of $V$ which has no common dual disk.
Choose any dual disks $D'$ and $E'$ of $D$ and $E$, and let $S$ and $T$ be the vertices of $\mathcal S(V, W; \Sigma)$ corresponding to the dual pairs $(D, D')$ and $(E, E')$ respectively.
We will show that there is no path joining $S$ and $T$.
Suppose, for contradiction, that there exists a sequence of dual pairs $(D, D') = (D_0, D'_0), (D_1, D'_1), \cdots, (D_n, D'_n) = (E, E')$ realizing a path $\mathcal L$ from $S$ to $T$ in $\mathcal S(V, W; \Sigma)$.

Define the simplicial map $\Phi_V: \mathcal S(V, W; \Sigma) \rightarrow \mathcal P'(V)$ sending the vertex of each dual pair to the vertex of the primitive disk in $V$ in the dual pair.
For example, $\Phi_V(S) = D$, $\Phi_V(T) = E$.
Then $\Phi_V(\mathcal L)$ is a path in $\mathcal P'(V)$ from $D$ to $E$.
Let $e$ be the edge in $\mathcal P(V)$ whose end vertices are $D$ and $E$.
Since $\{D, E\}$ has no common dual disk, the edge $e$ is not contained in the path $\Phi_V(\mathcal L)$.
Thus, the path $\Phi_V(\mathcal L)$ together with the edge $e$ in $\mathcal P(V)$
is itself a cycle or contains at least one cycle containing the edge $e$, which we denote by $\mathcal C$.
The cycle $\mathcal{C}$ is an $n$-cycle for some $n \geq 3$.
Further, among the primitive pairs which determines the edges of $\mathcal C$, $\{D, E\}$ is the only one having no common dual disk.

Suppose first that $\mathcal C$ is a $3$-cycle (in this case, we have $q = 2$ or $p = 2q+1$ necessarily, by Lemma \ref{lem:primitive_triples}).
Then $C$ bounds a $2$-simplex in $\mathcal P(V)$.
Among the three primitive pairs in the primitive triple representing the vertices of the $2$-simplex, $\{D, E\}$ is the only one having no common dual disk, which contradicts Lemma \ref{lem:primitive_triples} (4).

Suppose that $\mathcal C$ is an $n$-cycle with $n \geq 4$ (in this case, we have $(p, q) = (5, 2)$ necessarily by Lemma \ref{lem:primitive_triples}).
Then, as mentioned in the remark in Section \ref{sec:primitive_disk_complex}, there exist two $2$-simplices in $\mathcal D(V)$ such that exactly two edges of each of them are contained in $\mathcal C$ and further they have no common edges contained in $\mathcal C$.
Thus one of them, say $\Delta$, does not contain the edge $e$.
Two edges of $\Delta$ are contained in the cycle $\mathcal C$, but the other one is not contained in $\mathcal C$ and not even in $\mathcal P'(V)$ since $\mathcal P'(V)$ consists of tree components by Lemma \ref{lem:primitive_disk_complex} (3).
Consequently, among the three primitive pairs of the primitive triple representing the vertices of $\Delta$, exactly one pair has no common dual disk, which contradicts Lemma \ref{lem:primitive_triples} (4) again.

\bigskip

\noindent (3) If $L$ is $L(2, 1)$, then every edge in $\mathcal S(V, W; \Sigma)$ is contained in a unique cycle, which is a $4$-cycle.

\medskip

Let $S$ and $T$ be the end vertices of an edge in $\mathcal S(V, W; \Sigma)$ whose corresponding dual pairs are $(D, E')$ and $(E, E')$ respectively.
By Lemma \ref{lem:common_duals}, there exists exactly one more common dual disk $D'$ of the pair $\{D, E\}$, and $\{D', E'\}$ is a primitive pair in $W$.
Thus we have a $4$-cycle containing the edge joining $S$ and $T$ whose vertices correspond to $(D, E')$, $(E, E')$, $(E, D')$ and $(D, D')$ (see Figure \ref{fig:cycles} (a)).

\medskip

\begin{center}
\begin{overpic}[width=9cm, clip]{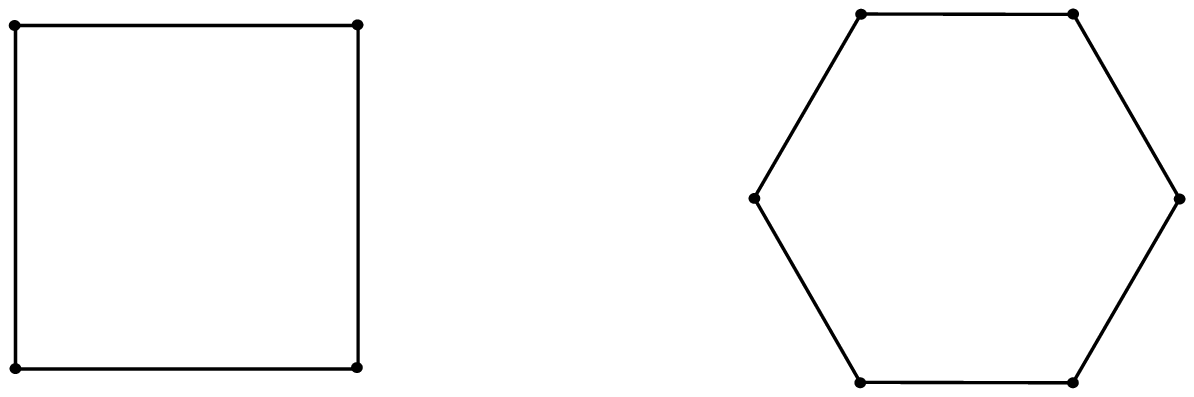}
  \linethickness{3pt}
    \put(0,90){$(D, E')$}
    \put(72,90){$(E, E')$}
    \put(72,5){$(E,D')$}
    \put(0,5){$(D, D')$}
    \put(45,-5){(a)}
    \put(155,90){$(D, E')$}
    \put(215,90){$(E, E')$}
    \put(245,47){$(E,F')$}
    \put(215,3){$(F,F')$}
    \put(155,3){$(F,D')$}
    \put(124,47){$(D,D')$}
    \put(197,-5){(b)}
\end{overpic}
\captionof{figure}{}
\label{fig:cycles}
\end{center}

Now suppose that there exists a cycle $\mathcal Z$ in $\mathcal S(V, W; \Sigma)$ which shares edge(s) with the $4$-cycle.
Suppose first that $\mathcal Z$ shares a single edge with the $4$-cycle.
We may assume that the end vertices of the edge are $S$ and $T$.
Then we write all the dual pairs of the vertices of $\mathcal Z$ consecutively, $(D, E') = (D_0, E'_0), (D_1, E'_1), \cdots, (D_n, E'_n) = (E, E')$.
For each $j \in \{0, 1, \cdots, n\}$, the pair $\{D_{j-1}, D_j\}$ cannot be $\{D, E\}$ since the only common dual disks of $\{D, E\}$ are $D'$ and $E'$.
Thus the image $\Phi_V(\mathcal Z)$ contains at least a cycle in $\mathcal P'(V)$, which contradicts Lemma \ref{lem:primitive_disk_complex} (1).
In case that $\mathcal Z$ shares more edges with the $4$-cycle, a similar argument holds to have the same contradiction.

\bigskip

\noindent (4) If $L$ is $L(3, 1)$, then every edge in $\mathcal S(V, W; \Sigma)$ is contained in a unique cycle, which is a $6$-cycle.

\medskip

Let $S$ and $T$ be the end vertices of the edge in $\mathcal S(V, W; \Sigma)$ whose corresponding dual pairs are $(D, E')$ and $(E, E')$ respectively.
By Lemma \ref{lem:common_duals} and \ref{lem:primitive_triples} (1) and (4), there exists a unique dual pair $(F, F')$ such that $\{D, E, F\}$ and $\{D', E', F'\}$ are primitive triples in $V$ and $W$ respectively, and $E'$, $F'$ and $D'$ are common dual disks of the pairs $\{D, E\}$, $\{E, F\}$ and $\{F, G\}$ respectively.
Thus the edge joining $S$ and $T$ is contained in the $6$-cycle, whose vertices correspond to $(D, E')$, $(E, E')$, $(E, F')$ $(F, F')$, $(F, D')$ and $(D, D')$ as in Figure \ref{fig:cycles} (b).
Further, by a similar argument to the case of $L(2, 1)$ in the above, we see that there exists no cycle which shares edge(s) with the $6$-cycle, using Lemma \ref{lem:common_duals} and Lemma \ref{lem:primitive_disk_complex} (2).

\bigskip

\noindent (5) and (6) If $L$ is $L(p, 1)$ with $p \geq 4$, then $\mathcal S(V, W; \Sigma)$ is a tree.
If $L$ is not $L(p, 1)$, then $\mathcal S(V, W; \Sigma)$ consists of infinitely many tree components.

\medskip
Since $\mathcal S(V, W; \Sigma)$ is $1$-dimensional, it suffices to show that there is no cycle in $\mathcal S(V, W; \Sigma)$.
Suppose there exists a cycle $\mathcal Z$ in $\mathcal S(V, W; \Sigma)$, and let $S$ and $T$ be the end vertices of an edge of $\mathcal Z$.
We may assume that $S$ and $T$ correspond to the dual pairs $(D, E')$ and $(E, E')$ respectively, and then write all the vertices of $\mathcal Z$ consecutively, $(D, E') = (D_0, E'_0), (D_1, E'_1), \cdots, (D_n, E'_n) = (E, E')$ as in the argument for $L(2, 1)$ in the above.
Then, for each $j \in \{0, 1, \cdots, n\}$, the pair $\{D_{j-1}, D_j\}$ cannot be $\{D, E\}$ since the only common dual disk of $\{D, E\}$ is $E'$.
Thus the image $\Phi_V(\mathcal Z)$ contains at least a cycle in $\mathcal P'(V)$, which contradicts Lemma \ref{lem:primitive_disk_complex} (1) and (3).
When $L$ is not $L(p, 1)$, $\mathcal S(V, W; \Sigma)$ consists of infinitely many tree components since any primitive disk is contained in infinitely many primitive pairs having no common dual disks.

\end{document}